\documentclass[12pt]{article}

\usepackage[dvips]{graphicx}
\usepackage{amsfonts, amssymb, amsmath, amsthm, psfrag,color}

\newcommand\bigzero{\makebox(0,0){\text{\huge0}}}

\DeclareRobustCommand\nlab{\nu}

\theoremstyle{theorem}
\newtheorem{thm}{Theorem}[section]
\newtheorem{lem}[thm]{Lemma}

\title{\vspace*{-1.0cm}
       \vspace*{1cm}{\bf \Large{Local Existence of a Classical Solution\\for Quasi-Linear
Hyperbolic Systems}}}

\author{Shih-Wei Chou\footnote{E-mail: swchou@scu.edu.tw}
	     \\
	     {\small Department of Finance Engineering and Actuarial Mathematics, Soochow University,}\\{\small Taipei 100006, Taiwan}\\ \\
        Ying-Chieh Lin\footnote{E-mail: linyj@nuk.edu.tw}
         \\
         {\small Department of Applied Mathematics, National University of Kaohsiung,}\\{\small Kaohsiung 811726, Taiwan}\\
         \\
        Naoki Tsuge\footnote{E-mail: ntsuge@hiroshima-u.ac.jp}
         \\
         {\small School of Engineering, Graduate School of Advanced Science and Engineering, Hiroshima University,}\\{\small Higashihiroshima, 739-8527, Japan}\\
         \\
         }

\date{}

\begin{document}

\maketitle

In this article, we give a proof of the local existence of the classical solutions to the quasilinear system
\begin{equation}
\label{quasilinear}
u_{t}+A(x,t,u)u_{x}=h(x,t,u)
\end{equation}
with initial data
\begin{equation}
\label{ID}
u(x,0)=\bar{u}(x),\qquad x\in [a,b],
\end{equation}
where we assume 
\begin{itemize}
\item[(A)] Each matrix $A(x, t, u)$ has $n$ real distinct eigenvalues. The functions $A:\mathbb{R}_+\times \mathbb{R}\times\mathbb{R}^n\rightarrow\mathbb{R}^n\times\mathbb{R}^n$ and $h:\mathbb{R}_+\times \mathbb{R}\times\mathbb{R}^n\rightarrow\mathbb{R}^n$ are continuously differentiable with respect to all variables and their derivatives are locally Lipschitz continuous with respect to all variables. In addition, we assume $\bar{u}$ is continuously differentiable.
\end{itemize}
For simplicity, we further assume that $A(x,t,u)$ is a diagonal matrix, i.e.,
\begin{equation*}
A(x,t,u)=\left(
\begin{array}{ccc}
\lambda_1(x,t,u) &        &               \\
                 &        & \bigzero      \\
                 &\ddots  &               \\
\bigzero         &        &               \\
                 &        &\lambda_n(x,t,u) 
\end{array}
\right).
\end{equation*}
Then we have the following theorem.

\begin{thm}
\label{local_existence}
Suppose that $A$, $h$, and $\bar{u}$ satisfy the assumption (A). Then There exist constants $\Lambda,T>0$ and a continuously differentiable function $u$ which is the unique classical solution of \eqref{quasilinear}, \eqref{ID} on the domain
\begin{equation}
\label{D}
\mathcal{D}=\mathcal{D}_{\Lambda,T}:=\{(x,t):t\in [0,T],\ a+\Lambda t\le x\le b-\Lambda t\}.    
\end{equation}
\end{thm}

\begin{proof}
To get the solution, we consider the sequence $u^{(\nu)}:\mathcal{D}\rightarrow\mathbb{R}^n$ such that $u^{(0)}(x,0)=\bar{u}(x)$ and, for $\nu\ge 1$, $u^{(\nu)}=(u^{(\nu)}_1,\cdots,u^{(\nu)}_n)$ is defined inductively by the solution of the semilinear problem
\refstepcounter{equation}
\begin{equation}
\label{def-u}\tag*{(\theequation)$_{\nlab}$}
u^{(\nu)}_{i,t}+\lambda_i(x,t,u^{(\nu-1)})u^{(\nu)}_{i,x}=h_i(x,t,u^{(\nu-1)}),\quad u^{(\nu)}(x,0)=\bar{u}(x),\quad i=1,\cdots,n, 
\end{equation}
where $h_i$ is the $i$-th component of $h$. 

1. We set
\begin{equation}
\label{C0}
C_0:=\max_{x\in [a,b]}|\bar{u}(x)|
\end{equation}
and choose $\Lambda$ in \eqref{D} to be
\begin{equation*}
\Lambda=\max\{|\lambda_i(x,t,u)|:\ t\in [0,1],\ x\in [a,b],\ |u|\le C_0+1,\ i=1,\cdots,n\}.
\end{equation*}
Under the choice of $\Lambda$, the set $\mathcal{D}$ becomes a domain of determinacy for \ref{def-u} provided that
\refstepcounter{equation}
\begin{equation}
\label{determinacy}\tag*{(\theequation)$_{\nlab}$}
T\le 1,\qquad |u^{(\nu-1)}(x,t)|\le C_0+1\quad\text{for}\ (x,t)\in\mathcal{D}.
\end{equation}
With the assumption (A), one can choose constants $C_1$ and $C_2$ such that
\begin{equation}
\label{ubar}
|\bar{u}'(x)|\le C_1,\quad |h(x,0,\bar{u}(x))-A(x,0,\bar{u}(x))\bar{u}'(x)|\le C_1\qquad\text{for all}\ x\in [a,b]
\end{equation}
and
\begin{equation}
\label{h_lambda}
|h_x|\le C_2,\quad |h_t|\le C_2,\quad |h_u|+|A_x|\le C_2,\quad |h_u|+|A_t|\le C_2,\quad |A_u|\le C_2    
\end{equation}
for all $t\in [0,1]$, $x\in [a,b]$, $|u|\le C_0+1$.

Let $\bar{Y}$ be the solution of the ODE
\begin{equation*}
Y'=C_2(1+nY)^2,\qquad Y(0)=C_1.
\end{equation*}
Then $\bar{Y}$ can be solved explicitly by
\begin{equation*}
\bar{Y}(t)=\frac{1}{n}\left(\frac{1}{(1+nC_1)^{-1}-nC_2t}-1\right)\quad\text{for all}\ t\in\left[0,\frac{1}{nC_2(1+nC_1)}\right).
\end{equation*}
We now choose $T>0$ small enough such that
\begin{equation}
\label{delta}
\int^{T}_{0}n\bar{Y}(t)dt\le 1.
\end{equation}

2. We now prove by induction that \ref{determinacy} holds together with 
\refstepcounter{equation}
\begin{equation}
\label{Lip}\tag*{(\theequation)$_{\nlab}$}
|u^{(\nu)}_x(x,t)|\le n\bar{Y}(t), \quad |u^{(\nu)}_t(x,t)|\le n\bar{Y}(t)\qquad\text{for all}\ (x,t)\in\mathcal{D},\ \nu\in\mathbb{N}\cup\{0\}.
\end{equation}
If \ref{determinacy} and \ref{Lip} holds for all $\nu\ge 0$, then we will prove the following:
\begin{itemize}
\item[(a)] $\mathcal{D}$ serves as a universal domain of determinacy for all problems \ref{def-u}.
\item[(b)] The Lipschitz constant of the functions $u^{(\nu)}$ is uniformly bounded on $\mathcal{D}$.
\end{itemize}
\ref{Lip} is obvious for $\nu=0$. Suppose that {\renewcommand\nlab{\nu-1}\ref{Lip}} is true. Then, by \eqref{C0} and \eqref{delta}, we have
\begin{equation}
\label{u^nu-1}
\aligned
|u^{(\nu-1)}(x,t)|&\le |u^{(\nu-1)}(x,0)|+\int^t_0 |u^{(\nu-1)}_t(x,s)|ds\\
&\le |\bar{u}(x)|+\int^t_0 n\bar{Y}(s) ds\le C_0+1\qquad\text{for}\ (x,t)\in\mathcal{D},
\endaligned
\end{equation}
which proves \ref{determinacy}. It follows from \cite[Theorem 3.6]{B} (see also \cite[Theorem 6]{D}) that problem \ref{def-u} admits a classical solution $u^{(\nu)}$ on $\mathcal{D}$. Moreorver, if we write $\mathfrak{v}:=u^{(\nu)}_x$, $\mathfrak{w}:=u^{(\nu)}_t$, and $\mathfrak{v}_i$, $\mathfrak{w}_i$ denote the $i$-th component of $\mathfrak{v}$, $\mathfrak{w}$ respectively, then they will be {\it broad solutions} (see \cite[p. 48]{B} for its definition) for the following systems
\begin{equation*}
\aligned
\mathfrak{v}_{i,t}+\lambda_i\left(x,t,u^{(\nu-1)}(x,t)\right)\mathfrak{v}_{i,x}
&=h_{i,x}+h_{i,u}\cdot u^{(\nu-1)}_x-\left(\lambda_{i,x}+\lambda_{i,u}\cdot u^{(\nu-1)}_x\right)\mathfrak{v},\quad i=1,\cdots,n,\\
\mathfrak{w}_{i,t}+\lambda_i\left(x,t,u^{(\nu-1)}(x,t)\right)\mathfrak{w}_{i,x}
&=h_{i,t}+h_{i,u}\cdot u^{(\nu-1)}_t-\left(\lambda_{i,t}+\lambda_{i,u}\cdot u^{(\nu-1)}_t\right)\mathfrak{w},\quad i=1,\cdots,n,
\endaligned
\end{equation*}
which implies
\begin{equation}
\label{u_x}
\frac{d}{d\tau}\left\{ \mathfrak{v}_{i}(x^{(\nu-1)}_i(\tau;x,t),\tau)\right\}
=h_{i,x}+h_{i,u}\cdot u^{(\nu-1)}_x-\left(\lambda_{i,x}+\lambda_{i,u}\cdot u^{(\nu-1)}_x\right)\mathfrak{v},\quad i=1,\cdots,n,
\end{equation}
\begin{equation}
\label{u_t}
\frac{d}{d\tau}\left\{ \mathfrak{w}_{i}(x^{(\nu-1)}_i(\tau;x,t),\tau)\right\}
=h_{i,t}+h_{i,u}\cdot u^{(\nu-1)}_t-\left(\lambda_{i,t}+\lambda_{i,u}\cdot u^{(\nu-1)}_t\right)\mathfrak{w},\quad i=1,\cdots,n,
\end{equation}
where $\tau\mapsto x^{(\nu-1)}_i(\tau;x,t)$ denotes the $i$-th characteristic curve related to $u^{(\nu)}$ passing through $(x,t)$; more precisely, $x^{(\nu-1)}_i(\tau;x,t)$ is the solution of 
\begin{equation}
\label{characteristics-1}
\frac{dX}{d\tau}=\lambda_i\left(X,\tau,u^{(\nu-1)}(X,\tau)\right),\qquad X(t)=x.
\end{equation}
Define
\begin{equation*}
Y(\tau):=\max\{|\mathfrak{v}_{i}(X,\tau)|,|\mathfrak{w}_{i}(X,\tau)|:\ X\in [a+\Lambda\tau,b-\Lambda\tau],\ i=1,\cdots,n\}.
\end{equation*}
In view of \eqref{ubar}-\eqref{h_lambda}, \eqref{u^nu-1}-\eqref{u_t}, and the induction hypothesis, we obtain that
\begin{equation*}
Y'(\tau)\le C_2[1+n\bar{Y}(\tau)+n^2\bar{Y}(\tau)Y(\tau)]<\bar{Y}'(\tau),\qquad Y(0)\le C_1,
\end{equation*}
provided that $Y(\tau)\le \bar{Y}(\tau)$. A simple comparison argument gives that $Y(\tau)\le \bar{Y}(\tau)$ for all $\tau\in [0,T]$. Thus, by induction, we complete the proof of \ref{Lip}.

3. Next, we show the uniform convergence of the sequence $u^{(\nu)}$ on $\mathcal{D}$. For simplicity, we write
\begin{equation}
\label{notations}
\mathfrak{u}^\nu=u^{(\nu)}-u^{(\nu-1)},\quad \lambda^{\nu-1}_i(x,t)=\lambda_i\left(x,t,u^{(\nu-1)}(x,t)\right),\quad 
h^{\nu-1}_i(x,t)=h_i\left(x,t,u^{(\nu-1)}(x,t)\right).
\end{equation}
From {\renewcommand\nlab{}\ref{def-u}}, we get that
\begin{equation*}
%\label{diff-u}
\mathfrak{u}^{\nu}_{i,t}+\lambda^{\nu-1}_i \mathfrak{u}^{\nu}_{i,x}=h^{\nu-1}_i-h^{\nu-2}_i-(\lambda^{\nu-1}_i-\lambda^{\nu-2}_i)u^{(\nu-1)}_{i,x},\quad i=1,\cdots,n,
\end{equation*}
and then
\begin{equation}
\label{diff-u-2}
\frac{d}{d\tau}\left\{ \mathfrak{u}^{\nu}_{i}(x^{(\nu-1)}_i(\tau;x,t),\tau)\right\}=h^{\nu-1}_i-h^{\nu-2}_i-(\lambda^{\nu-1}_i-\lambda^{\nu-2}_i)u^{(\nu-1)}_{i,x},\quad i=1,\cdots,n.
\end{equation}
From the assumption (A), there exists a constant $C_3$ such that
\begin{equation*}
%\label{Lip-1}
|A(x,t,u)-A(x,t,u')|\le C_3|u-u'|,\qquad |h(x,t,u)-h(x,t,u')|\le C_3|u-u'| 
\end{equation*}
provided that $(x,t)\in\mathcal{D}$, $|u|,|u'|\le C_0+1$. On the other hand, the bounds {\renewcommand\nlab{}\ref{Lip}} say that there exists a constant $C_4$ such that $|u^{(\nu)}_x|\le C_4$ for $(x,t)\in\mathcal{D}$ and $\nu\in\mathbb{N}\cup\{0\}$. We now consider the function
\begin{equation*}
Z_\nu(\tau):=\max\{|\mathfrak{u}^{\nu}_{i}(X,\tau)|:\ X\in [a+\Lambda\tau,b-\Lambda\tau],\ i=1,\cdots,n\}.
\end{equation*}
Then \eqref{diff-u-2} yields that
\begin{equation}
\label{Z_ineq}
Z'_\nu(\tau)\le C_3 n Z_{\nu-1}(\tau)+[C_3 n Z_{\nu-1}(\tau)]C_4,\quad Z_\nu(0)=0.
\end{equation}
Then the uniform convergence of the series $\sum Z_\nu(\tau)$, for $\tau\in [0,T]$, follows directly from the following lemma. Therefore, the sequence $u^{(\nu)}$ is uniformly convergent on $\mathcal{D}$.

\begin{lem}
Let $\{Z_\nu(\tau)\}_{\nu\ge 0}$ be a sequence of continuous, non-negative functions satisfying $Z_0(\tau)\le \bar{Z}e^{\alpha\tau}$. Suppose that $\alpha>0$ and $\beta,\bar{Z}\ge 0$. If $Z_\nu(\tau)$, $\nu\ge 1$, satisfy the following recurrent inequality
\begin{equation}
\label{recurrent}
Z_\nu(\tau)\le \int^\tau_0[\alpha Z_\nu(\eta)+\beta Z_{\nu-1}(\eta)]d\eta\qquad\text{for}\ \nu\ge 1,
\end{equation}
then, for $\nu\ge 1$,
\begin{equation}
\label{conv}
Z_\nu(\tau)\le\frac{(\beta \tau)^{\nu}e^{\alpha \tau}}{\nu!}\bar{Z}\quad\text{and}\quad \sum^\infty_{\nu=0}Z_{\nu}(\tau)=e^{(\alpha+\beta)\tau}\bar{Y}.
\end{equation}
\end{lem}

\noindent{\it Proof.}
Let
\begin{equation*}
W_\nu(\tau):=\int^\tau_0[\alpha Z_\nu(\eta)+\beta Z_{\nu-1}(\eta)]d\eta.
\end{equation*}
From \eqref{recurrent}, we deduce
\begin{equation*}
W'_\nu(\tau)=\alpha Z_\nu(\tau)+\beta Z_{\nu-1}(\tau)\le \alpha W_\nu(\tau)+\beta Z_{\nu-1}(\tau).
\end{equation*}
Multiplying $e^{-\alpha \tau}$ and integrating the resultant inequality on $[0,\tau]$, we have
\begin{equation*}
W_\nu(\tau)\le \beta\int^{\tau}_0 e^{\alpha(\tau-\eta)}Z_{\nu-1}(\eta)d\eta,
\end{equation*}
which together with \eqref{recurrent} gives that
\begin{equation*}
Z_\nu(\tau)\le \beta\int^{\tau}_0 e^{\alpha(\tau-\eta)}Z_{\nu-1}(\eta)d\eta.
\end{equation*}
Then one can easily obtain \eqref{conv} by induction on $\nu$.\hskip 1cm\qedsymbol

4. We now prove that $u^{(\nu)}_x$ is equicontinous on $\mathcal{D}$. Since $u^{(\nu)}_x$ and $u^{(\nu)}_t$ are uniformly bounded on $\mathcal{D}$ by {\renewcommand\nlab{}\ref{Lip}}, there exists a $C_5>0$ independent of $\nu$ such that
\begin{equation}
\label{equi_u}
|u^{(\nu)}(x,t)-u^{(\nu)}(x',t')|<C_5\delta_1\quad\text{if}\ |(x,t)-(x',t')|<\delta_1\ \text{and}\ (x,t),(x',t')\in \mathcal{D}
\end{equation}
for all $\delta_1>0$. 
%In this note, we frequently use $C$ as a constant, which may vary each time it appears.
Let $\varepsilon>0$ be given, we choose $0<\delta_2\le\delta_1/C_6\le\varepsilon/(4nC_7)$ and $(x,t)$, $(x',t')$ are any two fixed points in $\mathcal{D}$ such that $|(x,t)-(x',t')|<\delta_2$, where $C_6, C_7$ are constants determined later. By \eqref{characteristics-1}, we have
\begin{equation}
\label{characteristics-2}
\frac{d}{d\tau}x^{(\nu-1)}_i(t-\tau;x,t)=-\lambda^{\nu-1}_i\left(x^{(\nu-1)}_i(t-\tau;x,t),t-\tau\right)=:-\lambda_i,
\end{equation}
and
\begin{equation}
\label{characteristics-3}
\frac{d}{d\tau}x^{(\nu-1)}_i\left(t'-\frac{t'}{t}\tau;x',t'\right)
=-\frac{t'}{t}\lambda^{\nu-1}_i\left(x^{(\nu-1)}_i\left((t'-\frac{t'}{t}\tau;x',t'\right),t'-\frac{t'}{t}\tau\right)=:-\frac{t'}{t}\lambda'_i.
\end{equation}
For any fixed $\nu$, we use $x_i(\tau)$ and $x'_i(\tau)$ to denote $x^{(\nu-1)}_i(\tau;x,t)$ and $x^{(\nu-1)}_i(\tau'(\tau);x',t')$ for simplicity, where $\tau'(\tau):=\frac{t'}{t}\tau$.
Substracting \eqref{characteristics-3} from \eqref{characteristics-2}, we get that
\begin{equation}
\label{x-diff}
\frac{d}{d\tau}(x_i-x'_i)(t-\tau)=(\lambda'_i-\lambda_i)+\frac{t'-t}{t}\lambda'_i.
\end{equation}
The Lipschitz continuity of $\lambda_i$ and \eqref{equi_u} yields that
\begin{equation}
\label{lambda-diff}
|\lambda'_i-\lambda_i|\le C_8\left(|(x_i-x'_i)(t-\tau)|+|t-t'|\right),
\end{equation}
for some constant $C_8$. Applying the differential inequality on \eqref{x-diff}-\eqref{lambda-diff}, we have that there exists a constant chosen as $C_6$ such that
\begin{equation}
\label{x-diff-1}
|(x_i(\tau),\tau)-(x'_i(\tau),\tau'(\tau))|\le C_6\delta_2\le\delta_1\quad\text{for}\ \tau\in [0,t].
\end{equation}
And \eqref{x-diff-1} holds for any two points $(x,t)$, $(x',t')\in\mathcal{D}$ such that $|(x,t)-(x',t')|<\delta_2$ if $T$ is a sufficiently small constant independent of $\nu$ and $\delta_2$.

We now prove by induction that, if there exists a $\delta_3>0$ such that
\begin{equation}
\label{bar_u}
|\bar{u}'(x)-\bar{u}'(x')|\le\frac{\varepsilon}{2}\quad\text{for}\ |x-x'|<\delta_3\ \text{and}\ x,x'\in [a,b],
\end{equation}
then, for every $\nu\ge 0$, we have
\refstepcounter{equation}
\begin{equation}
\label{equi_nu}\tag*{(\theequation)$_{\nlab}$}
|u^{(\nu)}_x(x,t)-u^{(\nu)}_x(x',t')|<\varepsilon\quad\text{if}\ |(x,t)-(x',t')|<\delta_2\le \frac{\delta_3}{C_6}\ \text{and}\ (x,t),(x',t')\in \mathcal{D}
\end{equation}
provided that $T$ is small enough. \eqref{bar_u} follows directly from the assumption (A). In addition, it is easy to see that if \ref{equi_nu} holds, then we also have
\refstepcounter{equation}
\begin{equation}
\label{equi_nu_2}\tag*{(\theequation)$_{\nlab}$}
|u^{(\nu)}_x(x,t)-u^{(\nu)}_x(x',t')|<([C_7]+1)\varepsilon\quad\text{if}\ |(x,t)-(x',t')|<C_7\delta_2\ \text{and}\ (x,t),(x',t')\in \mathcal{D},
\end{equation}
where $[\cdot]$ denotes the greatest integer function.
\ref{equi_nu} for the case $\nu=0$ is obvious. We assume that {\renewcommand\nlab{\nu-1}\ref{equi_nu}} holds. To show that \ref{equi_nu} is also true, we write $v^{\nu}(x,t)=u^{(\nu)}_x(x,t)$ for simplicity. From \ref{def-u}, we have
\begin{equation}
\label{v_eq}
v^{\nu}_{i,t}+\lambda^{\nu-1}_i v^{\nu}_{i,x}=h^{\nu-1}_{i,x}+h^{\nu-1}_{i,u}\cdot v^{\nu-1}-\lambda^{\nu-1}_{i,x}v^{\nu}_i-\left(\lambda^{\nu-1}_{i,u}\cdot v^{\nu-1}\right) v^{\nu}_i,
\end{equation}
where $\lambda^{\nu-1}_i$ and $h^{\nu-1}_{i}$ are defined in \eqref{notations}. We obtain from \eqref{v_eq} that
\begin{equation}
\label{v_eq_x}
\frac{d}{d\tau}v^{\nu}_i(x_i(\tau),\tau)=\left\{ h^{\nu-1}_{i,x}+h^{\nu-1}_{i,u}\cdot v^{\nu-1}-\lambda^{\nu-1}_{i,x}v^{\nu}_i-\left(\lambda^{\nu-1}_{i,u}\cdot v^{\nu-1}\right) v^{\nu}_i\right\}(x_i(\tau),\tau).
\end{equation}
Since $\tau'(\tau)=\frac{t'}{t}\tau$ and
\begin{equation*}
(x'_i(\tau),\tau'(\tau))=\left(x_i\left(\frac{t'}{t}\tau\right),\frac{t'}{t}\tau\right)=(x_i(\tau'),\tau'),  
\end{equation*}
we also have
\begin{equation}
\label{v_eq_x'}
\aligned
\frac{d}{d\tau}v^{\nu}_i(x'_i(\tau),\tau'(\tau))&=\frac{t'}{t}\frac{d}{d\tau'}v^{\nu}_i(x_i(\tau'),\tau')\\
&=\frac{t'}{t}\left\{ h^{\nu-1}_{i,x}+h^{\nu-1}_{i,u}\cdot v^{\nu-1}-\lambda^{\nu-1}_{i,x}v^{\nu}_i-\left(\lambda^{\nu-1}_{i,u}\cdot v^{\nu-1}\right) v^{\nu}_i\right\}(x'_i(\tau),\tau'(\tau)).
\endaligned
\end{equation}
Let
\begin{equation}
\label{w}
w^{\nu}(\tau):=v^{\nu}(x_i(\tau),\tau)-v^{\nu}(x'_i(\tau),\tau'(\tau)).
\end{equation}
By \eqref{v_eq_x}-\eqref{w}, we get that
\begin{equation}
\label{w_eq}
\aligned
\frac{d}{d\tau}w^{\nu}_i(\tau)&=\Big\{\left[ h^{\nu-1}_{i,x}+h^{\nu-1}_{i,u}\cdot v^{\nu-1}-\lambda^{\nu-1}_{i,x}v^{\nu}_i-\left(\lambda^{\nu-1}_{i,u}\cdot v^{\nu-1}\right) v^{\nu}_i\right](x_i(\tau),\tau)\\
&\quad -\left[ h^{\nu-1}_{i,x}+h^{\nu-1}_{i,u}\cdot v^{\nu-1}-\lambda^{\nu-1}_{i,x}v^{\nu}_i-\left(\lambda^{\nu-1}_{i,u}\cdot v^{\nu-1}\right) v^{\nu}_i\right](x'_i(\tau),\tau'(\tau))\Big\}\\
&\quad +\frac{t-t'}{t}\left\{ h^{\nu-1}_{i,x}+h^{\nu-1}_{i,u}\cdot v^{\nu-1}-\lambda^{\nu-1}_{i,x}v^{\nu}_i-\left(\lambda^{\nu-1}_{i,u}\cdot v^{\nu-1}\right) v^{\nu}_i\right\}(x'_i(\tau),\tau'(\tau))\\
&=:I_1+I_2.
\endaligned
\end{equation}
The $C^1$ continuities of $h$, $\lambda$, and {\renewcommand\nlab{}\ref{determinacy}} give that there exists a constant chosen as $C_7$ such that
\begin{equation}
\label{I2-est}
|I_2|\le \frac{C_7|t-t'|}{t}\le \frac{C_7\delta_2}{t}\le \frac{\varepsilon}{4nt}.
\end{equation}
On the other hand, we write $I_1$ as
\begin{equation}
\label{I1}
I_1=J_1+J_2+J_3+J_4+J_5+J_6+J_7+J_8,
\end{equation}
where
\begin{equation*}
\aligned
J_1&:=\left\{h^{\nu-1}_{i,x}(x_i(\tau),\tau)-h^{\nu-1}_{i,x}(x'_i(\tau),\tau'(\tau))\right\},\\
J_2&:=\left\{h^{\nu-1}_{i,u}(x_i(\tau),\tau)-h^{\nu-1}_{i,u}(x'_i(\tau),\tau'(\tau))\right\}\cdot v^{\nu-1}(x_i(\tau),\tau),\\
J_3&:=h^{\nu-1}_{i,u}(x'_i(\tau),\tau'(\tau))\cdot w^{\nu-1}(\tau),\\
J_4&:=-\left\{\lambda^{\nu-1}_{i,x}(x_i(\tau),\tau)-\lambda^{\nu-1}_{i,x}(x'_i(\tau),\tau'(\tau))\right\}v^{\nu}_i(x_i(\tau),\tau),\\
J_5&:=-\lambda^{\nu-1}_{i,x}(x'_i(\tau),\tau'(\tau))w^{\nu}_i(\tau),\\
J_6&:=-\left\{\left(\lambda^{\nu-1}_{i,u}(x_i(\tau),\tau)-\lambda^{\nu-1}_{i,u}(x'_i(\tau),\tau'(\tau))\right)\cdot v^{\nu-1}(x_i(\tau),\tau)\right\} v^{\nu}_i(x_i(\tau),\tau),\\
J_7&:=-\left\{\lambda^{\nu-1}_{i,u}(x'_i(\tau),\tau'(\tau))\cdot w^{\nu-1}(\tau)\right\} v^{\nu}_i(x_i(\tau),\tau),\\
J_8&:=-\left\{\lambda^{\nu-1}_{i,u}(x'_i(\tau),\tau'(\tau))\cdot v^{\nu-1}(x'_i(\tau),\tau'(\tau))\right\} w^{\nu}_i(\tau).
\endaligned
\end{equation*}
Applying {\renewcommand\nlab{}\ref{determinacy}}, \eqref{equi_u}, \eqref{x-diff-1}, {\renewcommand\nlab{\nu-1}\ref{equi_nu_2}}, and the assumption (A), we obtain that
\begin{equation}
\label{est1}
|J_1|,|J_2|,|J_4|,|J_6|\le C_9(\delta_2+\delta_1)\le C_{10}\varepsilon,\qquad
|J_3|,|J_7|\le C_{10}\varepsilon,.
\end{equation}
for some constants $C_9,C_{10}$. The uniform boundedness of $u^{(\nu)}$ and $u^{(\nu)}_x$ gives that
\begin{equation}
\label{est2}
|J_5|,|J_8|\le C_{11}|w^{\nu}_i|,
\end{equation}
for some constant $C_{11}$. Using the differential inequality on \eqref{w_eq}-\eqref{est2} together with \eqref{bar_u}, we have 
\begin{equation*}
%\label{w-est}
|u^{(\nu)}_{i,x}(x,t)-u^{(\nu)}_{i,x}(x',t')|=|w^{(\nu)}_i(t)|\le \frac{\varepsilon}{n}\quad\text{for}\ t\in [0,T]
\end{equation*}
if $T$ is small enough, which is valid for $i=1,\cdots,n$. 
Therefore, \ref{equi_nu} is also true. By induction, we get the equicontinuity of $u^{(\nu)}_x$.

5. We are in a position to prove the local existence of classical solution for problem \eqref{quasilinear}, \eqref{ID}. Since $u^{(\nu)}$ is uniformly convergent on $\mathcal{D}$, we let
\begin{equation}
\label{u}
u(x,t):=\lim_{\nu\rightarrow\infty}u^{(\nu)}(x,t)\quad\text{for}\ (x,t)\in\mathcal{D}.
\end{equation}
On the other hand, since $u^{(\nu)}_x$ is uniformly bounded and equicontinuous on $\mathcal{D}$, the Arzela-Ascoli theorem says that there exists a subsequence $\{\nu_k\}$ of $\mathbb{N}$ such that $u^{(\nu_k)}_x$ is uniformly convergent on $\mathcal{D}$. Thus, we get that
\begin{equation}
\label{u_x,u_t}
u_x(x,t)=\lim_{k\rightarrow\infty}u^{(\nu_k)}_x(x,t)
\end{equation}
exists and is continuous on $\mathcal{D}$. From {\renewcommand\nlab{}\ref{def-u}} and the Lipschitz continuity of $h_i$, $\lambda_i$, we find that
\begin{equation}
\label{lim_u^nu_t}
\aligned
\lim_{k\rightarrow\infty}u^{(\nu_k)}_t(x,t)&=h_i\left(x,t,\lim_{k\rightarrow\infty}u^{(\nu_k)}\right)-\lambda_i\left(x,t,\lim_{k\rightarrow\infty}u^{(\nu_k)}\right)\lim_{k\rightarrow\infty}u^{(\nu_k)}_x(x,t)\\
&=h_i(x,t,u)-\lambda_i(x,t,u)u_x.
\endaligned
\end{equation}
Since the convergence in \eqref{lim_u^nu_t} is uniform with respect to $x$ and $t$, we have
\begin{equation}
\label{def_u_t}
\lim_{k\rightarrow\infty}u^{(\nu_k)}_t(x,t)=\left(\lim_{k\rightarrow\infty}u^{(\nu_k)}(x,t)\right)_t=u_t(x,t)
\end{equation}
exists is continuous on $\mathcal{D}$. Combining \eqref{lim_u^nu_t} and \eqref{def_u_t}, we prove that $u$ is a classical solution for problem \eqref{quasilinear}, \eqref{ID} on $\mathcal{D}$.
The uniqueness of classical solution problem \eqref{quasilinear}, \eqref{ID} follows directly from \cite[Theorem 1]{D}.
\end{proof}

%%%%%%%%%%%%%%%%%    BIBLIOGRAPHY    %%%%%%%%%%%%%%%%%%%%%%%%%%%

%%%%%%%%%%%%%%%%%%%%%%%%%%%%%%%%%%%%%%%%%%%%%%%%%%%%%%%%%%%%%%%%%%%%%%%%%


\vskip 1cm
\begin{thebibliography}{99}

\bibitem{B} Alberto Bressan, {\it Hyperbolic Systems of Conservation Laws}, Oxford university press (2000)

\bibitem{D} Avron Douglis, Some existence theorems for hyperbolic systems of partial differential equations in two independent variables, {\it Commun. Pure Appl. Math.} {\bf 5}, 119-154 (1952) 

\end{thebibliography}
\end{document}